\documentclass[12pt]{amsart}
\usepackage{amsfonts}
\usepackage{amssymb}
\usepackage[letterpaper, left=2.5cm, right=2.5cm, top=2.5cm,
bottom=2.5cm,dvips]{geometry}
\usepackage{verbatim}
\newtheorem{theorem}{Theorem}
\theoremstyle{plain}

\newtheorem{conjecture}{Conjecture}
\newtheorem{corollary}{Corollary}

\newtheorem{definition}{Definition}

\newtheorem{lemma}{Lemma}

\numberwithin{equation}{section}

\DeclareMathOperator{\col}{col}
\DeclareMathOperator{\modulo}{mod}

\begin{document}
	\title[Hat chromatic number of graphs]{Hat chromatic number of graphs}
	
	\author{Bart\l omiej Bosek}
	\address{Theoretical Computer Science Department, Faculty of Mathematics and Computer Science, Jagiellonian
		University, ul. Prof. Stanis\l awa \L ojasiewicza 6, 30-348 Krak\'{o}w, Poland}
	\email{bosek@tcs.uj.edu.pl}
	
	\author{Andrzej Dudek}
	\address{Department of Mathematics, Western Michigan University, Kalamazoo, MI}
	\email{ andrzej.dudek@wmich.edu}
	
	\author{Micha{\l} Farnik}
\address{Faculty of Mathematics and Computer Science,
	Jagiellonian University, ul. Prof. Stanis\l awa \L ojasiewicza 6, 30-348 Krak\'{o}w, Poland}
\email{michal.farnik@gmail.com}

	\author{Jaros\l aw Grytczuk}
	\address{Faculty of Mathematics and Information Science, Warsaw University
		of Technology, 00-662 Warsaw, Poland}
	\email{j.grytczuk@mini.pw.edu.pl}
	
	\author{Przemys\l aw Mazur}
\address{Wayve Technologies, 30 Station Road, Cambridge CB1 2RE, United Kingdom}
\email{przemek@wayve.ai}

\thanks{The first and fourth author are partially supported by the grant of Narodowe Centrum Nauki, grant number 2017/26/D/ST6/00264.}

	\begin{abstract}
We study the hat chromatic number of a graph defined in the following way: there is one player at each vertex of a
loopless graph $G$, an adversary places a hat of one
of $K$ colors on the head of each player, two players can see each other's hats if and only if they are
at adjacent vertices. All players simultaneously try to guess the color of their hat. The players cannot communicate but collectively determine a strategy before the hats are placed. The hat chromatic number, $\mu(G)$, is the largest number $K$ of colors such that the players are able to fix a strategy that will ensure that for every possible placement of hats at least one of the guesses correctly.

We compute $\mu(G)$ for several classes of graphs, for others we establish some bounds. We establish connections between the hat chromatic number, the chromatic number and the coloring number. We also introduce several variants of the game: with multiple guesses, restrictions on allowed strategies or restrictions on colorings. We show examples how the modified games can be used to obtain interesting results for the original game.
	\end{abstract}
	
	\maketitle
	
	\section{Introduction}
	
	We study a graph coloring problem inspired by the following hat guessing puzzle. There are $n$ players (we call them \emph{Bears}) sitting around and looking at each other. There is also an adversary (we call her \emph{Demon}) who suddenly puts colored hats on their heads. Each Bear can see all hats except his own. After a while, each Bear writes down on a piece of paper hypothetical color of his own hat. No communication with other Bears is allowed, though before the play Bears may fix some strategy. They win collectively against Demon if at least one of them guesses correctly. Otherwise, when none of them guessed correctly, Demon is the winner. Notice that Demon, as a supernatural creature, can read Bears' minds. So, she knows their strategy before the play. What is the maximum number of hats' colors (depending on $n$) for which there is a strategy guaranteeing Bears' win?
	
	This puzzle has a natural generalization for arbitrary graphs. Bears are sitting at the vertices of a given graph $G$, and each Bear can see only colleagues occupying neighboring vertices. So, his guess depends only on color configuration appearing in his neighborhood. The original puzzle concerns the case of complete graph $K_n$. We ask more generally: what is the maximum number of hats' colors for which there is a strategy guaranteeing Bears' win on a graph $G$?
	
	We denote this number by $\mu(G)$, and call it the \emph{hat chromatic number} of a graph $G$. Rigorous definition will appear in the next section, where, as a warm up, we also solve the original puzzle by proving that $\mu(K_n)=n$. Another family of graphs for which the problem is  completely solved are cycles. A surprising theorem due to Szczechla \cite{Szcz} asserts that $\mu(C_n)=3$ if $n=3k$ or $n=4$, and $\mu(C_n)=2$ in all other cases. The proof is quite involved. It is also known that $\mu(T) \leqslant 3$ for any tree $T$ (first shown in \cite{BHKL}). In fact in \cite{KL} it is shown that for a connected graph $\mu(G)\leq 2$ if and only if $G$ is a tree or $G$ contains a unique cycle $C_n$ where $n$ is not divisible by $3$ and $n\geq 5$. A Lov\'{a}sz Local Lemma argument shows that $\mu(G)\leqslant e(\Delta+1)$ for graphs with maximum degree $\Delta$.

In \cite{BHKL} it was shown that for a complete bipartite graph we have $\mu(K_{k-1,k^{k^n}})\geq k$. In \cite{ABST} it was shown that for a sufficiently large $r$-partite graph $\mu(K_{n,\ldots,n})\geq n^{\frac{r-1}{r}-o(1)}$.
	
	In this paper we prove more results and formulate several conjectures on the number $\mu(G)$, focusing on graphs with bounded density. In particular, we prove that $\mu(G)$ is bounded for graphs of bounded genus and sufficienlty large girth (depending on genus).
	
	\section{Warm up}
	
	Let us start with a simple solution to the original puzzle with $n$ Bears on a clique $K_n$.
	\begin{theorem}\label{th_clique}
		Every clique satisfies $\mu(K_n)=n$.
	\end{theorem}
	\begin{proof}
	Denote Bears by $B_1,B_2,\dots, B_n$, and assume that the set of colors is $\mathbb Z_n$. First we show that Bears have a winning strategy. Let $x_i$ denote the color of hat obtained by $B_i$. Suppose that the total sum of colors chosen by Demon is $S=x_1+x_2+\dots +x_n$. Now each Bear $B_i$ imagines that $S\equiv i(\modulo n)$, and guesses accordingly by writing a missing term as a hypothetical color of his own hat. More precisely, the strategy of Bear $B_i$ is given by an expression $$F_i=i-(x_1+\dots +x_{i-1}+x_{i+1}+\dots+x_n).$$ Since there are $n$ Bears and $n$ possible values of $S$ in $\mathbb Z_n$, we must have $x_i-F_i=0$ for at least one $i$, which means that at least one Bear guesses correctly.
	
	To see that Bears cannot win if the number of colors exceeds $n$, we apply a simple probabilistic argument. Let $k$ be a fixed number of colors. Assume that Bears have fixed their deterministic strategies. This means that a guess of each Bear is uniquely determined by hat colors of his neighbors. Suppose that Demon distributes colored hats randomly, choosing a color for each Bear independently with uniform probability. Let $A_i$ denote the event that the $i$-th Bear guesses correctly. Clearly, $\Pr(A_i)=1/k$. Thus, the probability that at least one Bear guesses correctly satisfies
	$$\Pr\left(\bigcup_{i=1}^{n}A_i\right)\leqslant \sum_{i=1}^{n}\Pr (A_i)=\frac{n}{k}.$$
	This implies that for $k>n$, with positive probability none of the Bears guesses correctly, and therefore Demon is the winner.
	\end{proof}

		\section{Notation and definitions}
		
	Now, we give a formal definition of the hat chromatic number $\mu (G)$. Let $G$ be a graph on the set of vertices $V=\{1,2,\dots,n\}$. Let $\Gamma$ be a fixed set of colors. Suppose that each vertex $i$ is assigned an $n$-ary function $F_i(x_1,x_2,\dots, x_n)$ mapping $\Gamma^n$ to $\Gamma$. We assume however, that $F_i$ depends only on those variables $x_j$ for which $j$ is adjacent to $i$ in $G$. In other words, a value of $F_i$ stays constant under any changes on coordinates corresponding to vertices not adjacent to $i$. Such functions $F_i$ will be called \emph{strategies} on a graph $G$.
	
	Consider now a system of equations
	\begin{equation}\label{system}
	F_i(x_1,x_2,\dots,x_n)=x_i
	\end{equation}
	 for $i=1,2,\dots,n$, where $F_i$'s are some strategies on a graph $G$. Suppose that for every substitution of elements from $\Gamma$ for variables $x_i$, at least one of equations in the system (\ref{system}) is satisfied. Then, the number $\mu(G)$ is defined as the largest integer $k=|\Gamma|$, for which such a system of equations exists for a graph $G$. We call it the \emph{hat chromatic number} of a graph $G$.
	
	In other words, if $k\geqslant \mu (G)+1$, then for every system of equations (\ref{system}) there is a substitution of elements form $\Gamma$ for variables $x_i$ such that none of the equations in the system (\ref{system}) holds. Such substitution will be called a \emph{demonic coloring} of a graph $G$ with respect to fixed strategies $F_i$. Therefore the hat chromatic number $\mu (G)$ can be defined equivalently as the least positive integer $k$ such that for every strategies $F_i$ on a graph $G$, there exists a demonic coloring of $G$ using $k+1$ colors.
	
	We can make this definition even more algebraic in the following way. Assume that $\Gamma$ is a set of $k$-th roots of unity in the field of complex number $\mathbb{C}$. For every strategy $F_i$ we may find a multivariable complex polynomial $P_i$ representing $F_i$ over $\Gamma$. This means that for every $a\in \Gamma^n$ we have $P_i(a)=F_i(a)$. The system of equations (\ref{system}) is then equivalent to a single polynomial equation
		\begin{equation}\label{single polynomial}
	P(x_1,x_2,\dots,x_n)=\prod_{i=1}^n(x_i-P_i(x_1,x_2,\dots,x_n))=0,
	\end{equation}
	over $\Gamma^n$. Since $x^k=1$ for every $x\in\Gamma$, we may reduce polynomial $P$ to a new polynomial $\tilde{P}$, in which every variable appears with exponent at most $k-1$. Then the hat chromatic number $\mu(G)$ is the least integer $k-1$ such that for every polynomial $P$, as defined above, its reduced version $\tilde{P}$ is non zero.

	\section{Probabilistic bounds}

We start with a simple proof that the hat chromatic number $\mu(G)$ is bounded for graphs of bounded maximum degree. It will be sufficient to use the symmetric version of the Lov\'{a}sz Local Lemma.

\begin{lemma}
Let $A_1,A_2,\dots,A_n$ be events in some probability space. Assume that $\Pr(A_i)\leqslant p$, and no event $A_i$ depends on more than $d$ other events. If $ep(d+1)\leqslant 1$, then $$\Pr\left(\bigcap_{i=1}^{n}\overline{A_i}\right)>0.$$
\end{lemma}
This lemma gives almost immediately the following result.
\begin{theorem}
Every graph $G$ of maximum degree $\Delta$ satisfies $\mu (G)\leqslant e(\Delta+1)$.
\end{theorem}
\begin{proof}
Let $k$ be the number of colors. Suppose that Bears fixed their strategies and Demon plays randomly. Let $A_i$ denote the event that the $i$-th Bear guessed correctly. Then $\Pr(A_i)=1/k$. It is not hard to check that each event $A_i$ is mutually independent of all other events $A_j$, except those for which $i$ and $j$ are adjacent in $G$. This can be explained as follows: if colors of all neighbors of $i$ are fixed, then the color guessed by the $i$-th Bear is uniquely determined. The event $A_i$ reduces then to randomly picking this color by Demon, which certainly does not depend on what happens in the remaining part of the graph. So, we may apply Lemma 1 for $p=1/k$ and $d=\Delta$, which immediately gives the assertion of the theorem.
\end{proof}

\begin{lemma}\label{lem:ub}
	Let $k$ be a positive integer and let $G=(V,E)$ be a graph of order $n$. Assume that there is a partition of $V = V_1\cup\cdots\cup V_{\ell}$ such that $V_i$ is an independent set for each $i\in[\ell]$. Then, if
	\begin{equation}\label{eq:lem1}
	\ell   - \sum_{i=1}^\ell \left( \frac{k-1}{k}\right)^{|V_i|} < 1,
	\end{equation}
	then $\mu(G) \leqslant k-1$.
\end{lemma}

\begin{proof}
	Assume that \eqref{eq:lem1} holds and $\mu(G) \geqslant k$. That means that with $k$ colors Bears can always win.
	Fix a strategy for each of the $n$ players. For a fixed $i$ there are exactly $(k-1)^{|V_i|}k^{n-|V_i|}$ colorings for which no player in $V_i$ guesses his color. Thus, there are $k^n - (k-1)^{|V_i|}k^{n-|V_i|}$ colorings for which at least one of the players in $V_i$ guesses his color. Consequently, there are at most
	\[
	\sum_{i=1}^{\ell} \left( k^n - (k-1)^{|V_i|}k^{n-|V_i|} \right)
	\]
	colorings such that there is a player who guesses his color. If this number is less then the total number of colorings (which is equivalent to~\eqref{eq:lem1}), then the adversary can choose a coloring for which none of the players will guess the color of his hat, a contradiction.
\end{proof}

\begin{theorem}\label{thm:ub}
	Let $k$ be a positive integer and let $G=(V,E)$ be a graph of order $n$ with chromatic number $h$. Then, if
	\[
	k > \frac{1}{1 - \left( 1 - \frac{1}{h} \right)^{h/n}},
	\]
	then $\mu(G) \leqslant k-1$.
\end{theorem}

\begin{proof}
	Let $V = V_1\cup\dots\cup V_{h}$ be a color partition. By Lemma~\ref{lem:ub} it suffices to show that 
	\[
	h   - \sum_{i=1}^h \left( \frac{k-1}{k}\right)^{|V_i|} < 1.
	\]
	Since $f(x) = \left( \frac{k-1}{k}\right)^x$ is a convex function, Jensen's inequality yields that
	\[
	\sum_{i=1}^h \left( \frac{k-1}{k}\right)^{|V_i|} \ge h \left( \frac{k-1}{k}\right)^{n/h}
	\]
	and consequently
	\begin{equation}\label{eq:calc1}
	h   - \sum_{i=1}^h \left( \frac{k-1}{k}\right)^{|V_i|} 
	\le h - h \left( \frac{k-1}{k}\right)^{n/h} 
	= h - h \left( 1-\frac{1}{k}\right)^{n/h}. 
	\end{equation}
	Finally, observe that by assumption $1 - \left( 1 - \frac{1}{h} \right)^{h/n} > \frac{1}{k}$
	and so $1 - \frac{1}{k} > \left( 1 - \frac{1}{h} \right)^{h/n}$,
	which implies in~\eqref{eq:calc1} that
	\[
	h   - \sum_{i=1}^h \left( \frac{k-1}{k}\right)^{|V_i|} < h - h \left( 1 - \frac{1}{h} \right) = 1.
	\]
\end{proof}

\begin{corollary}
	Every graph $G$ of order $n$ satisfies $\mu(G)\leqslant n$. Moreover, if $G$ is not a clique, then $\mu(G) < n$.
\end{corollary}

\begin{proof}
	If $G=K_n$, then $h = h(G) = n$ and
	\[
	\frac{1}{1 - \left( 1 - \frac{1}{h} \right)^{h/n}} = n < n+1
	\]
	implying that $\mu(K_n) \leqslant n$. Otherwise, if $h = h(G) < n$, then observe that $f(x) = \left( 1 - \frac{1}{x} \right)^{x/n}$ is an increasing function and so $f(h) < f(n) = 1-\frac{1}{n}$. Consequently,
	\[
	\frac{1}{1 - \left( 1 - \frac{1}{h} \right)^{h/n}}
	< \frac{1}{1 - \left(1-\frac{1}{n} \right)} = n.
	\]
\end{proof}

\begin{corollary}
	Let $k$ be a positive integer and let $G=(V,E)$ be a graph of order $n$ with chromatic number $h=h(n)\geqslant 2$.  Then, 
	\[
	\mu(G) \leqslant \frac{n}{h\ln\left( \frac{h}{h-1} \right)}
	\]
	for sufficiently large $n$.
\end{corollary}

\begin{proof}
	Clearly, 
	\[
	\left( 1 - \frac{1}{h} \right)^{h/n} 
	= \exp\left\{\frac{h}{n} \ln\left( 1 - \frac{1}{h} \right)\right\}
	= \exp\left\{-\frac{h}{n} \ln\left( \frac{h}{h-1} \right)\right\}.
	\]
	Now observe that $\frac{h}{n} \ln\left( \frac{h}{h-1} \right)=o(1)$ (as $n$ tends to infinity). Indeed, let $f(h) = h \ln\left( \frac{h}{h-1} \right)$. Then, it is easy to check that $f(h)$ is a decreasing function and so 
	\[
	2\ln(2) = f(2) \geqslant f(h) > \lim_{h\to\infty} f(h) =1.
	\] 
	Thus, $\lim_{n\to \infty} \frac{f(h)}{n} = 0$.
	
	Let $x = \frac{h}{n} \ln\left( \frac{h}{h-1} \right)$. Hence, $x>0$ and $x=o(1)$. Since $e^{-x} \leqslant  1 - x + \frac{x^2}{2} < 1$, we get
	\begin{align*}
	\frac{1}{1 - \left( 1 - \frac{1}{h} \right)^{h/n}}
	&= \frac{1}{1 -e^{-x}}
	\leqslant \frac{1}{x-\frac{x^2}{2}}
	= \frac{1}{x} + \frac{1}{2-x} 
	= \frac{1}{x} + \frac{1}{2} + \frac{x}{4-2x} 
	= \frac{1}{x} + \frac{1}{2} + o(1). 
	\end{align*}
	Thus, Theorem~\ref{thm:ub} yields that $\mu(G) \leqslant \frac{1}{x}$ for large $n$.
\end{proof}

\begin{corollary}
	For almost all graphs $G$ of order~$n$ we have
	\[
	(2+o(1))\log_2 n \leqslant \mu(G) \leqslant n - (1+o(1))\log_2 n.
	\]
\end{corollary}
\begin{proof}
	The lower bound is trivial, since almost all graphs contain a clique of order $(2+o(1))\log_2 n$. The upper bound will follow from the previous corollary. First, since $\ln(1+x) \geqslant x-\frac{x^2}{2} + \frac{x^3}{4}\geqslant 0$ for $0\leqslant x \leqslant 0.44$, we obtain that
	\begin{align*}
	\ln\left( \frac{h}{h-1} \right)
	=\ln\left( 1+\frac{1}{h-1} \right)
	&\geqslant \frac{1}{h-1} - \frac{1}{2(h-1)^2} + \frac{1}{4(h-1)^3}\\
	&= \frac{2}{2h-1} + \frac{1}{4(h-1)^3(2h-1)}
	\geqslant \frac{2}{2h-1}
	\end{align*}
	and consequently
	\[
	h \ln\left( \frac{h}{h-1} \right)
	\geqslant \frac{2h}{2h-1}
	= 1 + \frac{1}{2h-1}.
	\]
	Next recall that for almost all graphs $h(G) \leqslant \frac{n}{2\log_2 n} \cdot \frac{1}{1-\varepsilon_n}$, where $\varepsilon_n = \frac{4\log_2\log_2 n}{\log_2 n}$. Thus,
	\begin{align*}
	\mu(G) &\leqslant \frac{n}{h\ln\left( \frac{h}{h-1} \right)}
	\leqslant \frac{n}{1+ \frac{1}{\frac{n}{(1-\varepsilon_n)\log_2 n}-1} }\\
	&= n-(1-\varepsilon_n)\log_2 n
	= n - \log_2 n + 4\log_2\log_2 n =  n - (1+o(1))\log_2 n,
	\end{align*}
	as required.
\end{proof}

\section{Graphs of bounded density}
Any nonempty set of the form $C=C_1\times C_2\times \dots \times C_n$ will be called a \emph{cube}. A set $C_i$ is then called the \emph{$i$-th component} of the cube $C$. By $\Gamma_k$ we denote the set of $k$ colors ($k$-th complex roots of unity, for instance).

\subsection{Trees}
We start with a simple proof that $\mu(T)=2$ for every tree $T$ with at least two vertices. The lower bound $\mu(T)\geqslant2$ follows from $\mu(K_2)=2$ and an easy observation that $\mu(H)\leqslant\mu(G)$ whenever $H$ is a subgraph of $G$. For the upper bound we will need the following definition. Let $T$ be a tree with root $r$. Let $F_r$ be any strategy for $r$. A color $d\in \Gamma_k$ is said to be \emph{dominant} for $F_r$ if $F_r^{-1}(d)$ contains a cube whose each component have size at least $(k-1)$.

\begin{lemma}\label{Lemma dominant}
Let $T$ be a tree with root $r$. If $k\geqslant 3$, then for any strategy $F_r$ there exists at most one dominant color in $\Gamma_k$.
\end{lemma}
\begin{proof}
Suppose that there are two distinct dominant colors $a$ and $b$ for the root $r$. Let $C_a$ and $C_b$ denote the corresponding cubes contained in $F_r^{-1}(a)$ and $F_r^{-1}(b)$, respectively. Since each component of every cube is of size at least $k-1$, these components must overlap. Hence, $C_a\cap C_b\neq\emptyset$.  So, there is an element $x\in \Gamma_k^n$ such that $F_r(x)=a$ and $F_r(x)=b$. This means that $a=b$.
\end{proof}

Now we may prove the aforementioned result for the hat chromatic number of trees.

\begin{theorem}\label{Theorem Trees}
Let $T$ be a tree with root $r$, and let $F_i$ be fixed strategies on $T$. If $k\geqslant 3$, then for every color $c\in \Gamma_k$ which is not dominant for $F_r$, there is a demonic coloring $f$ such that $f(r)=c$. In consequence, every tree $T$ with at least two vertices satisfies $\mu (T)=2$.
\end{theorem}

\begin{proof}
	Let $T$ be a tree on the set of vertices $V=\{1,2,\dots,n\}$. We use induction on $n$. If $n=1$, then $F_r$ is a constant function of one variable, that is, $F_r(x)=d$ for some $d\in \Gamma_k$. Notice that $d$ is the unique dominant color for $F_r$. So, taking $f(r)=c$ for any color $c$ different than $d$ defines a demonic coloring of $T$.
	
	For the inductive step, let $r_1,r_2,\dots, r_t$ denote the neighbors of $r$ in $T$. Let $T_1,T_2,\dots, T_t$ denote connected components of $T-r$. Choose $r_i$ to be the root of $T_i$. Now, let $c$ be any non-dominant color for $F_r$. Denote by $F_{r_i}^{(c)}$ the restriction of strategy $F_{r_i}$ obtained by putting $x_r=c$. Assume that $d_i$ is a dominant color for $F_{r_i}^{(c)}$ (or any color if such does not exist). Put $A_{r_i}=\Gamma_k \setminus \{d_i\}$ for $i=1,2,\dots,t$, and $A_j=\Gamma_k$ for all other $j\neq r_i$. Since $c$ is not dominant for $F_r$, the whole cube $A_1\times\dots\times A_n$ cannot be contained in $F_r^{-1}(c)$. Hence, there must exist $a_i\in A_i$ such that $F_r(a_1,a_2,\dots,a_n)\neq c$. By Lemma \ref{Lemma dominant}, none of the colors $a_{r_i}$ is dominant for $F_{r_i}^{(c)}$. Hence, by inductive hypothesis, there exist demonic colorings $f_i$ of trees $T_i$ such that $f_i(r_i)=a_{r_i}$. Now we may define a coloring $f$ by taking $f(r)=c$ and $f(v)=f_i(v)$ for all other vertices $v$ of $T$. Clearly, $f$ is a demonic coloring of $T$, and the proof is complete.
\end{proof}

\subsection{Multiple guesses}

Consider now a modified hat guessing game in which we allow each Bear to guess $s$ times, where $s\geqslant 1$ is a fixed integer. In other words, each Bear picks a subset of $s$ colors, and they win if at least one Bear hit an actual color of his hat. Let $\mu_s(G)$ denote the analog of the hat chromatic number $\mu(G)$. We can easily generalize the results concerning trees. Let $T$ be a tree with root $r$. Let $F_r$ be any strategy for $r$. A color $d\in \Gamma_k$ is said to be \emph{$s$-dominant} for $F_r$ if $F_r^{-1}(d)=\{x\in\Gamma_k^n\ :\ d\in F_r(x)\}$ contains a cube whose each component have size at least $(k-s)$.

\begin{lemma}\label{Lemma s-dominant}
Let $T$ be a tree with root $r$. If $k>s(s+1)$, then there are at most $s$ $s$-dominant colors in $\Gamma$ for any strategy $F_r$.
\end{lemma}
\begin{proof}
Suppose that there are $s+1$ distinct dominant colors $a_1,\ldots,a_{s+1}$ for the root $r$. Let $C^i=C^i_1\times\ldots\times C^i_n$ denote the cubes contained in $F_r^{-1}(a_i)$. For each $j\in{1,\ldots,n}$ we have
$$\left|\bigcap_{i=1}^{s+1} C^i_j\right|\geqslant k-\sum_{i=1}^{s+1}(k-|C_i|)\geqslant k-s(s+1)>0.$$
Hence, $\bigcap C^i\neq\emptyset$.  So, there is an element $x\in \Gamma_k^n$ such that $\{a_1,\ldots,a_{s+1}\}\subset F_r(x)$, a contradiction.
\end{proof}

\begin{theorem}\label{th_Trees-s}
	Every tree $T$ satisfies $\mu_s(T)\leqslant s(s+1)$ for every $s\geqslant 1$.
\end{theorem}

\begin{proof}
	Let $T$ be a tree on the set of vertices $V=\{1,2,\dots,n\}$. We use induction on $n$. If $n=1$, then $F_r$ is a constant function of one variable, that is, $F_r(x)=D$ for some $D\in \mathcal{P}(\Gamma_k)$ of cardinality $s$. Notice that $D$ is the set of dominant colors for $F_r$. So, taking $f(r)=c$ for any color $c\notin D$ defines a demonic coloring of $T$.
	
	For the inductive step, let $r_1,r_2,\dots, r_t$ denote the neighbors of $r$ in $T$. Let $T_1,T_2,\dots, T_t$ denote connected components of $T-r$. Choose $r_i$ to be the root of $T_i$. Now, let $c$ be any non-dominant color for $F_r$. Denote by $F_{r_i}^{(c)}$ the restriction of strategy $F_{r_i}$ obtained by putting $x_r=c$. Assume that $D_i$ are sets of $s$ colors containing all dominant colors for $F_{r_i}^{(c)}$. Put $A_{r_i}=\Gamma_k \setminus D_i$ for $i=1,2,\dots,t$, and $A_j=\Gamma_k$ for all other $j\neq r_i$. Since $c$ is not dominant for $F_r$, the whole cube $A_1\times\dots\times A_n$ cannot be contained in $F_r^{-1}(c)$. Hence, there must exist $a_i\in A_i$ such that $F_r(a_1,a_2,\dots,a_n)\neq c$. By Lemma \ref{Lemma s-dominant}, none of the colors $a_{r_i}$ is dominant for $F_{r_i}^{(c)}$. Hence, by inductive hypothesis, there exist demonic colorings $f_i$ of trees $T_i$ such that $f_i(r_i)=a_{r_i}$. Now we may define a coloring $f$ by taking $f(r)=c$ and $f(v)=f_i(v)$ for all other vertices $v$ of $T$. Clearly, $f$ is a demonic coloring of $T$, and the proof is complete.
\end{proof}

The following theorem motivates introducing the multiple guessing variant of the hat chromatic number.

\begin{theorem}\label{th_partition}
Let $G$ be a connected graph, and let $V=A\cup B$ be a partition of the vertex set of $G$. Let $d=\max \{|N(v)\cap A|\ :\ v\in B\}$. Then $\mu_s(G)\leqslant \mu_{s_1}(G[B])$, where $s_1=s(\mu_s(G[A])+1)^d$.
\end{theorem}

\begin{proof}
Let $K>\mu_{s_1}(G[B])$ be the number of colors, let $K_1=\mu_s(G[A])+1$ and take $\Gamma_{K_1}\subset \Gamma_K$. Let $F_i$ be fixed strategies for the graph $G$, $K$ colors and $s$ guesses. We will construct a demonic coloring for the $F_i$, note that we will use only colors from $\Gamma_{K_1}$ to color vertices in $A$.

First let us construct strategies $F_i^B$ for $G[B]$ with $K$ colors and $s_1=sK_1^d$ guesses. Let $v\in B$ and let $y\in\Gamma_K^{|N(v)\cap B|}$ be a coloring of the neighbors of $v$ in $B$. We set $F_v^B(x)$ to be a fixed subset of $\Gamma_K$ of cardinality $sK_1^d$ containing the sets $F_v(x,y)$ for all $x\in\Gamma_{K_1}^{|N(v)\cap A|}$. Such a set can be chosen because $|N(v)\cap A|\leqslant d$.

Since $K>\mu_{s_1}(G[B])$ we can find a demonic coloring $\phi_B$ of $G[B]$ with respect to $F_i^B$.

Now let us construct a strategy $F_i^A$ for $G[A]$ with $K_1$ colors and $s$ guesses. Let $v\in A$ and $x\in\Gamma_{K_1}^{|N(v)\cap A|}$. We set $F_v^A(x)$ to be a fixed subset of $\Gamma_{K_1}$ containing $F_v(x,\phi_B)\cap\Gamma_{K_1}$.

Since $K_1>\mu_s(G[A])$ we can find a demonic coloring $\phi_A$ of $G[A]$ with respect to $F_i^A$.

We claim that $\phi=(\phi_A,\phi_B)$ is a demonic coloring for the $F_i$. Indeed, if $v\in A$ and $\phi(v)\in F_v(\phi)$ then $\phi_A(v)\in F^A_v(\phi_A)$, a contradiction with the choice of $\phi_A$. If $v\in B$ and $\phi(v)\in F_v(\phi)$ then $\phi_B(v)\in F^B_v(\phi_B)$, a contradiction with the choice of $\phi_B$.
\end{proof}

We exhibit the usefulness of Theorem \ref{th_partition} by showing the following two results.

\begin{theorem}
Every graph $G$ of genus $\gamma$ and sufficiently large girth (depending on $\gamma$) satisfies $\mu_s(G)\leqslant(s^2+s)(s^2+s+1)$, in particular $\mu(G)\leqslant 6$. 
\end{theorem}

\begin{proof}
According to \cite{ACKKR}, Lemma $5.1$ it is folklore that for every surface $S$ there is a girth $\gamma$ such that every graph $G$ of girth at least $\gamma$ embedded in $S$ has a partition $V(G)=A\cup B$ such that $G[B]$ is a tree and $A$ is two independent in $G$. Since $A$ is two independent we have $|N(v)\cap A|\leqslant 1$ for all $v\in B$ and $\mu_s(G[A])=s$. Thus from Theorem \ref{th_partition} and Theorem \ref{th_Trees-s} we obtain $\mu_s(G)\leqslant\mu_{s(s+1)}(G[B])\leqslant s(s+1)(s(s+1)+1)$.
\end{proof}

\begin{theorem}
	Every graph $G$ has a subdivision $S$ satisfying $\mu_s(S)\leqslant s(s+1)^2$, in particular $\mu(S)\leqslant 4$.
\end{theorem}

\begin{proof}
We construct $S$ by subdividing each edge. Note that $S$ is a bipartite graph with $V(S)=A\cup B$, where $A=V(G)$ and $B$ is the set of vertices introduced in the subdivision. Obviously $|N(v)\cap A|=2$ for all $v\in B$ and $\mu_s(G[A])=s$. Thus by Theorem \ref{th_partition} we have $\mu_s(S)\leqslant\mu_{s(s+1)^2}(S[B])=s(s+1)^2$.
\end{proof}

\subsection{Restricted demonic colorings}

Another way of modifying the hat chromatic number is by restricting the set of allowable strategies that Bears can use. Let $\mathcal{M}$ be a family of strategies specified by some property, for instance. By $\mu_{\mathcal{M}}(G)$ we denote the maximum size of color set $\Gamma$ for which Bears can win with using only strategies from $\mathcal{M}$.

In our next result we give a bound for $\mu_{\mathcal{B}}(G)$ for bounded density graphs, where $\mathcal{B}$ is a family of \emph{bi-polar} strategies defined as follows.

\begin{definition}
We call a strategy $F_i$ \emph{bi-polar with respect to an order} if for all $j\geqslant i$ and all partial colorings $(x_1,\ldots,x_{j-1})\in \Gamma_K^{j-1}$ we have: for all $x_j\in \Gamma_K$ the set $F_i(\{(x_1,\ldots,x_{j-1},x_j)\}\times \Gamma_K^{n-j})$ is either equal $\Gamma_K$ or is a singleton, moreover for any $y\in \Gamma_K$ there may by at most one $x_j\in \Gamma_K$ such that $F_i(\{(x_1,\ldots,x_{j-1},x_j)\}\times \Gamma_K^{n-j})=\{y\}$.

We call a strategy \emph{bi-polar} if it is bi-polar with respect to all orders.
\end{definition}

The ``sum modulo $K$'' strategy from Theorem \ref{th_clique} is an example of a strategy that is bi-polar with respect to all orders.  A strategy for a Bear $v$: ``if any of your neighbors has a red hat, then answer ``red'', otherwise say the color of the hat on Bear $w$'' is bi-polar with respect to those orders for which $w$ is the last of the neighbors of $v$.

Recall that a \emph{coloring number} of a graph $G$, denoted by $\col (G)$, is the least integer $k$ for which there is a linear order of the vertices of $G$ such that each vertex has at most $k-1$ neighbors appearing earlier in the order.

\begin{theorem}
	Every graph $G$ satisfies $\mu_{\mathcal{B}}(G)\leqslant \col (G)$.
\end{theorem}
\begin{proof}
Let $K>\col (G)$ be the number of colors and $F_i$ be the strategies. We will construct a demonic coloring inductively by extending a partial coloring of the first $t$ vertices. 

Let $(x_1,\ldots,x_t)$ be the partial coloring of the first $t$ vertices such that Bears $v_i$ for $i=1,\ldots,t$ either cannot guess yet because $F_i(\{(x_1,\ldots,x_t)\}\times\Gamma_K^{n-t})=\Gamma_K$ or guess incorrectly because $F_i(\{(x_1,\ldots,x_t)\}\times\Gamma_K^{n-t})=\{c\}$ for some color $c\neq x_i$. We will find a suitable color $x_{t+1}$ for the $t+1$-st vertex. If $v_i$ is not an neighbor of $v_{t+1}$ then $F_i$ does not depend on $x_{t+1}$. Let $v_i$ be a neighbor of $v_{t+1}$. Since $F_i$ is bi-polar there is at most one possible color $c_i$ such that $F_i(\{(x_1,\ldots,x_t,c_i)\}\times\Gamma_K^{n-t-1})=\{x_i\}$. Furthermore, $F_{t+1}$ does not depend on the color of $v_{t+1}$ so $F_{t+1}(\{(x_1,\ldots,x_t,x_{t+1})\}\times\Gamma_K^{n-t-1})=\Gamma_K$ or $F_{t+1}(\{(x_1,\ldots,x_{t+1})\}\times\Gamma_K^{n-t-1})=\{c_{t+1}\}$ regardless of $x_{t+1}$. We fix $x_{t+1}$ so that is distinct from $c_i$ for $i\leqslant t$ such that $v_i$ is a neighbor of $v_{t+1}$ and distinct from $\{c_{t+1}\}$. This can be done because $K>\col (G)$.
\end{proof}

Impressed by the theorem above one may wonder whether $\mu(G)\leqslant \col (G)$. In Theorem \ref{th_Skn} we will show a family of graphs $G_k$ with $\col (G)=k+1$ and $\mu(G)\geqslant 2^k$.

\section{Variable color sets}

In this section we consider another variation of our hat guessing game. This time we assume that each Bear $B_i$ has its private set of colors $\Gamma_i$. A strategy $F_i$ is then a function mapping the product $\Gamma_1\times \Gamma_2\times \dots \times \Gamma_n$ into $\Gamma_i$ (depending only on coordinates corresponding to the neighbors of the vertex $i$). A sequence $(a_1,a_2,\dots,a_n)$ is called \emph{winning} if Bears have winning strategies for any sets of colors $\Gamma_i$, with $|\Gamma_i|=a_i$. Otherwise, the sequence $(a_1,a_2,\dots,a_n)$ is called \emph{loosing}.

As a direct corollary from Theorem \ref{Theorem Trees} we have the following:

\begin{theorem}
For every tree $T$, the sequence $(2,3,\dots,3)$ is loosing.
\end{theorem}

On the other hand we have:

\begin{theorem}
If $T$ is a tree with degree sequence $(d_1,d_2,\dots,d_n)$, then the sequence $(2^{d_1},2^{d_2},\dots,2^{d_n})$ is winning for $T$.
\end{theorem}
\begin{proof}
We prove the theorem by induction on $n$. Assume that $v$ is a leaf of $T$ and $w$ a neighbor of $v$ with $\deg(w)=d$. Consider $T'=T-v$. Let $\Gamma_w(T')=\{1,\ldots, 2^{d-1}\}$ and $\Gamma_w(T)=\{1,\ldots, 2^d\}$ be the set of colors of $w$ for $T'$ and $T$, respectively. Let $\Gamma_v(T)=\{1,2\}$. By the induction hypothesis the Bears have winning strategies for $T'$, we modify them to strategies for $T$ as follows:
\begin{itemize}
\item the Bears on $T'$ that are not neighbors of $w$ apply for $T$ the same strategy as for $T'$, 
\item the Bears on $T'$ that are neighbors of $w$ pretend that $w$ is colored with $\lceil c(w)/2\rceil$ instead of $c(w)$ and use the strategy from $T'$,
\item the Bear $v$ answers $2$ if $c(w)$ is odd and $1$ if $c(w)$ is even, 
\item the Bear $w$ obtains an answer $d$ from his strategy for $T'$ and answers $2(d-1)+c(v)$.
\end{itemize} 

To check that the strategy is winning assume that $c$ is a coloring of $T$. It induces a coloring of $T'$ and at least one of the Bears guesses correctly on $T'$. If it is not $w$ then that Bear also guesses correctly on $T$. If it is $w$ then it means that $d=\lceil c(w)/2\rceil$. That means that $w$ answers correctly if $v$ and $w$ have the same parity and $v$ answers correctly if $v$ and $w$ have distinct parity.
\end{proof}

We would also like to cite a result by Szczechla (\cite{Szcz}, Corollary 8):

\begin{theorem}
	For every cycle $C_n$, the sequence $(4,3,\dots,3)$ is loosing. In consequence, $\mu(C_n)\leqslant 3$.
\end{theorem}

Now let us make a slight generalization of the game. After choosing a graph $G$ and the number $K$ of colors let us also choose a set $A\subset \Gamma_K^n$ of \emph{admissible colorings}. The Bears know the set $A$ before determining their strategy and the Demon must choose a coloring from $A$. Obviously for $A=\Gamma_K^n$ we obtain the standard game. Moreover, if $A\subset\Gamma_K^i\times\{c\}\times\Gamma_K^{n-i-1}$ then $F_i=c$ is a winning strategy.

By $\mu_a(G)$ we denote the largest integer $t$ such that for all $K$ and for all subsets $A\subset\Gamma_K^n$ of cardinality at most $t$ the Bears have a winning strategy in the game with the set of admissible colorings $A$. Obviously $\mu_a(G)<(\mu(G)+1)^n$.

The usefulness of admissible colorings is presented in Theorem \ref{th_Skn}, before proving it we need the following:

\begin{lemma}\label{lem_Kn_a}
$\mu_a(K_n)\geqslant 2^n-1$
\end{lemma}
\begin{proof}
We prove the Lemma by induction on $n$. For $n=1$ we have one Bear and one admissible coloring.

Now assume $\mu_a(K_{n-1})\geqslant 2^{n-1}-1$ and let $A$ be the set of admissible colorings on $K_n$ with $|A|\leqslant 2^n-1$. Let $\pi:\Gamma_K^n\rightarrow\Gamma_K^{n-1}$ be the projection on the first $n-1$ coordinates and let $B=\{b\in\Gamma_K^{n-1}\ :\ |\pi^{-1}(b)\cap A|>1\}$. Note that $|A|\geqslant 2|B|$, in particular $|B|\leqslant 2^{n-1}-1$. Moreover, for every $c\in\pi(A)\setminus B$ there is a unique $f(c)\in\Gamma_K$ such that $(c,f(c))\in A$.

We define the following strategy: Bears $1,\ldots,n-1$ play on $K_{n-1}$ disregarding the color of the hat on Bear $n$, they use a winning strategy for the set of admissible colorings $B$. Bear $n$ uses the strategy $F_n$ such that $F_n(c)=f(c)$ if $c\in\pi(A)\setminus B$ and $F_n(c)=1$ otherwise.

Observe that the strategy defined above is a winning strategy for the Bears. Indeed, let $d\in A$ be a coloring. If $\pi(d)\in B$ then one of the Bears $1,\ldots,n-1$ guesses correctly. If $\pi(d)\notin B$ then Bear $n$ guesses correctly.
\end{proof}

Let $S_{k,n}$ denote the $k-star$, i.e., the graph obtained by replacing the degree $n$ vertex of $K_{1,n}$ by a clique on $k$ vertices.

\begin{theorem}\label{th_Skn}
$\mu(S_{k,n})\geqslant\mu_a(K_k)+1\geqslant 2^k$ for $n$ large enough.
\end{theorem}
\begin{proof}
Let $K=2^k$ be the number of colors. Let $B$ be the set $K-1$ element subsets of $\Gamma_K^k$. Take $n=|B|$ and $V(S_{k,n})=V_1\cup V_2$, where $S_{k,n}[V_1]$ is a clique on $k$ vertices and $S_{k,n}[V_2]$ is the empty graph on $n$ vertices. Every Bear in $V_2$ adopts the following strategy: if the coloring of $k$ vertices in $V_1$ belongs to the set the Bear represents, then it says the number of that coloring (in lexicographical order), otherwise the Bear says $K$.

The Bears in $V_1$ compute (each one individually but all with the same data) the set $A$ of possible colorings of $V_1$ for which none of the Bears in $V_2$ answers correctly. Note that $|A|\leqslant K-1$, indeed suppose that $A'\subset A$ is a set of $K-1$ elements. Since the Bear in $V_2$ corresponding to $A'$ does not guess correctly for any of the colorings in $A'$ his hat must have color $K$, so he would guess correctly if $V_1$ were colored with a coloring not in $A'$.

After computing $A$ the Bears in $V_1$ adopt a fixed winning strategy for a game on $K_k=S_{k,n}[V_1]$ with admissible colorings $A$, which is supplied by Lemma \ref{lem_Kn_a}.
\end{proof}

\section{Conjectures}

We conclude the paper with several conjectures:

\begin{conjecture}
	There is a function $F$ such that every graph $G$ satisfies $\mu(G)\leqslant F(\col (G))$.
\end{conjecture}

If such $F$ exists then by theorem \ref{th_Skn} we have $F(k)\geq 2^{k-1}$.

\begin{conjecture}
	If $G$ is a graph with degree sequence $(d_1,d_2,\dots,d_n)$, then the sequence $(d_1,d_2,\dots,d_n)$ is winning for $G$.
\end{conjecture}

Note that for $C_5$ the sequence $(d_1+1,\ldots,d_5+1)$ is not winning.

\begin{conjecture}
	Every graph $G$ with maximum degree $\Delta$ satisfies $\mu(G)\leqslant \Delta+1$.
\end{conjecture}

\begin{conjecture}
	Every planar graph $G$ satisfies $\mu(G)\leqslant 4$.
\end{conjecture}

\begin{conjecture}
	Every graph $G$ satisfies $\chi(G)\leqslant \mu(G)+1$.
\end{conjecture}

Let $h(G)$ denote the Hadwiger number of $G$ (the order of a largest clique minor in $G$).
\begin{conjecture}
Every graph $G$ satisfies $\mu(G)\leqslant h(G)$.
\end{conjecture}


\begin{thebibliography}{9}

\bibitem{ACKKR} M. O. Albertson, G. G. Chappell, H. A. Kierstead, A. K\"undgen, R. Ramamurthi, Coloring with no $2$-colored $P_4$'s, Electron.  J. Combin. 11 (2004) Research Paper 26, 13 pp. (electronic).

\bibitem{ABST} N. Alon, O. Ben-Eliezer, C. Shangguan, I. Tamo, The hat guessing number of graphs, arXiv:1812.09752v1 [math.CO], 2018.

\bibitem{AlonSpencer} N. Alon, J. H. Spencer, The Probabilistic Method, John Wiley and Sons, Inc. 2000.

\bibitem{Borodin} O. V. Borodin, On acyclic colorings of planar graphs, Discrete Math. 25 (1979) 211--236.

\bibitem{BHKL} S. Butler, M. Hajianghayi, R. Kleinberg, T. Leighton, Hat guessing games, SIAM J.~Discrete Math. 22 (2008), 592--605.

\bibitem{KL} K. Kokhas, A. Latyshev, For Which Graphs the Sages Can Guess Correctly the Color of at Least One Hat, J. Math. Sci. 236 (2019), 503-520.

\bibitem{Szcz} W. Szczechla, The three-colour hat guessing game on the cycle graphs, arXiv:1412.3435 [math.CO], 2014.
\end{thebibliography}
\end{document}